\documentclass[twoside,a4paper]{amsart}
\usepackage{amssymb,latexsym}
\input xypic
  \xyoption{all}

\usepackage{amsmath, amsthm, amsfonts}

\usepackage{pifont,pxfonts,txfonts}
\usepackage{yfonts}
\usepackage{bbm}
\usepackage{calrsfs}
\usepackage{empheq}
\usepackage{color}
\usepackage[colorlinks,linkcolor=blue,anchorcolor=blue,citecolor=blue]{hyperref}
\usepackage{amsmath, amstext, amsbsy, amscd}
\usepackage[mathscr]{eucal}
\usepackage{times}
\usepackage{amsmath, amsthm, amsfonts,extarrows}

\newtheorem{theorem}{Theorem}[section]
\newtheorem{proposition}[theorem]{Proposition}
\newtheorem{lemma}[theorem]{Lemma}
\newtheorem{remark}[theorem]{Remark}

\newtheorem{corollary}[theorem]{Corollary}

\newtheorem{conjecture}[theorem]{Conjecture}

   \newcommand{\ba}{\begin{eqnarray}}
   \newcommand{\na}{\end{eqnarray}}
   \newcommand{\ban}{\begin{eqnarray*}}
   \newcommand{\nan}{\end{eqnarray*}}

\newcommand{\bA}{{\mathbb A}}
\newcommand{\bC}{{\mathbb C}}
\newcommand{\bD}{{\mathbb D}}
\newcommand{\bF}{{\mathbb F}}

\newcommand{\bP}{{\mathbb P}}
\newcommand{\bQ}{{\mathbb Q}}

\newcommand{\bZ}{{\mathbb Z}}

\newcommand{\cO}{{\mathcal O}}
\newcommand{\cF}{{\mathcal F}}
\newcommand{\cI}{{\mathcal I}}
\newcommand{\fz}{{\mathfrak z}}

\newcommand{\sI}{{\mathscr I}}

  \newcommand{\<}{\langle}
  \renewcommand{\>}{\rangle}

\newcommand{\suml}{\sum\limits}
\newcommand{\prodl}{\prod\limits}

\begin{document}

\title{A flop formula for Donaldson-Thomas invariants}

\author[Hua-Zhong Ke]{Hua-Zhong Ke}
\maketitle

\begin{center}
\emph{Department of Mathematics, Sun Yat-Sen University, Guangzhou,  510275, China}
\emph{kehuazh@mail.sysu.edu.cn}
\end{center}

{\bf Abstract:} Let $X$ and $X'$ be nonsingular projective $3$-folds related by a flop of a disjoint union of $(-2)$-curves. We prove a flop formula relating the Donaldson-Thomas invariants of $X$ to those of $X'$, which implies some simple relations among BPS state counts. As an application, we show that if $X$ satisfies the GW/DT correspondence for primary insertions and descendants of the point class, then so does $X'$. We also propose a conjectural flop formula for general flops.


{\bf MR(2010) Subject Classification:} 14N35

\date{\today}

\tableofcontents

\section{Introduction}

The Donaldson-Thomas theory of a nonsingular projective $3$-fold $X$ counts the number of stable sheaves on $X$ \cite{DT,Th}. In particular, when considering ideal sheaves of curves, the theory gives virtual numbers of embedded curves in $X$. Another curve counting theory on $X$ is the much studied Gromov-Witten theory, which essentially counts stable maps from curves with marked points to $X$. In \cite{MNOP1,MNOP2}, Maulik, Nekrasov, Okounkov, and Pandharipande proposed a remarkable conjecture that the Gromov-Witten theory of $X$ is equivalent to the Donaldson-Thomas theory of $X$ in a subtle way. This suggests that many phenomenon in one theory have counter parts in the other theory.

The above mentioned curve counting theories are deformation invariant. A fundamental problem in Gromov-Witten theory is to investigate the transformation of Gromov-Witten invariants under birational surgeries \cite{Ru}. The first breakthrough is the work of Li and Ruan \cite{LR}, who showed that, for $3$-folds, the primary Gromov-Witten theories are invariant under general flops. It is also important to study the effect of biraional surgeries on Donaldson-Thomas theory. Hu and Li \cite{HL} used the degeneration formula to understand the change of Donaldson-Thomas invariants under flops of a disjoint union of $(-1,-1)$ curves which are all numerally equivalent. For general flops between Calabi-Yau $3$-folds, Toda \cite{T2} used the categorical method to established a flop formula (see also \cite{Ca}). 

In this paper, we prove a flop formula in Donaldson-Thomas theory for flops of a disjoint union $(-2)$-curves, and derive some interesting relations on BPS state counts. As an application, we give positive evidence for the conjectural GW/DT correspondence. Here an embedded curve in a $3$-fold is a $(-2)$-curve \cite{Re} if it is a nonsingular rational curve with normal bundle of type $(-1,-1)$ or $(0,-2)$. Our flop formula generalizes the result of Hu and Li \cite{HL}, since a $(-1,-1)$-curve is a $(-2)$-curve, and we do not assume that the curves are numerically equivalent.

Throughout this paper, let $X$ and $X'$ be nonsingular projective $3$-folds over $\bC$, which are related by a flop $f:X\dashrightarrow X'$ of some contraction \cite{Ko}. Then $f$ is a birational map, and it is biregular outside of a subvariety of codimension two in $X$, called the center of $f$. The center of $f$ is a disjoint union of trees of rational curves, and it has a neighborhood with trivial canonial bundle. We have a natural isomorphism of groups
\ban
\cF:H_2(X,\bZ)\xrightarrow{\cong} H_2(X',\bZ),
\nan
defined as follows. For any $\beta\in H_2(X,\bZ)$, we can choose a real $2$-dimensional pseudo-submanifold $\Sigma$ representing $\beta$ in $X$, which lies in the complement of the center of $f$. Now $\cF\beta$ is represented by $f(\Sigma)$ in $X'$, which lies in the complement of the center of $f^{-1}$. Similarly, by considering Poincar\'e duals of classes of degree$\geqslant3$, we also have an isomorphism
\ban
H^{\geqslant 3}(X,\bQ)\rightarrow H^{\geqslant 3}(X',\bQ),
\nan
which can be extended to an isomorphism of cohomology groups
\ban
H^{*}(X,\bQ)\xrightarrow{\cong} H^{*}(X',\bQ),
\nan
by requiring this isomorphism to preserve the Poincar\'e pairing. The isomorphism will also be denoted by $\cF$ by abuse of notation. Let $Cen(f)$ be the subgroup of $H_2(X,\bZ)$ generated by the cycles of irreducible curves in the center of $f$. The main result of this paper is the following.

\begin{proposition}\label{mainresult}
Let $f$ be a flop of a disjoint union of $(-2)$-curves. Suppose that $\gamma_1,\cdots,\gamma_m\in H^{*}(X,\bQ)(m\geqslant 0)$ have supports away from the center of $f$, and $d_1,\cdots,d_m\in\bZ_{\geqslant 0}$. Then we have
\ba
\frac{\suml_{\beta\in H_2(X,\bZ)}v^\beta Z'_{DT}(X;q|\prodl_{i=1}^m\tilde\tau_{d_i}(\gamma_i))_\beta}{\suml_{\beta\in Cen(f)}v^{\beta} Z'_{DT}(X;q|)_{\beta}}&=&\frac{\suml_{\beta\in H_2(X,\bZ)}v^\beta Z'_{DT}(X';q|\prodl_{i=1}^m\tilde\tau_{d_i}(\cF\gamma_i))_{\cF\beta}}{\suml_{\beta\in Cen(f)}v^{\beta} Z'_{DT}(X';q|)_{\cF\beta}},\label{flop}\\
Z'_{DT}(X;q|)_\beta&=&Z'_{DT}(X';q|)_{-\cF\beta},\qquad\forall\beta\in Cen(f).\label{center}
\na
\end{proposition}
\begin{remark}
We remark that we can choose the support of $\gamma_i$ away from the center of $f$ if deg$\gamma_i>2$.
\end{remark}

We sketch the proof of Proposition \ref{mainresult}, the detail of which will be given in Section 3. By a beautiful result of Reid \cite{Re}, we can decompose the flop $f$ of $(-2)$-curves into a sequence of blow-ups of $(-2)$-curves followed by a sequence of blow-downs. Since blow-ups can be described in terms of semi-stable degenerations, it follows that we can use the degeneration formula \cite{LW} and the absolute/relative correspondence \cite{HLR,MP} to relate invariants of $X$ to those of the blow-up of $X$ (see \eqref{X}). Therefore, in principle, the Donaldson-Thomas invariants of $X$ can be related to those of $X'$. Due to the denominators in \eqref{flop}, we need to understand the transformation of the invariants attached to classes in $Cen(f)$ under blow-ups. To this end, we give a detailed analysis of the change of effectiveness of classes in $Cen(f)$ under blow-up (see Lemma \ref{1stlemma} and \ref{2ndlemma}).

Proposition \ref{mainresult} relates the Donaldson-Thomas invariants of $X$ to those of $X'$ in a nontrivial way. In \cite{HHKQ} and \cite{Ke}, we obtained some blow-up formulae for Gromov-Witten and stable pair theories which contain some extra factors, and we discovered that these formulae imply some interesting relation among BPS state counts. In this paper, we consider the change of BPS state counts of Donaldson-Thomas theory under flops. Proposition \ref{mainresult} implies the following simple flop formulae for BPS state counts.
\begin{corollary}\label{BPS}
Let $f$ be a flop of a disjoint union of $(-2)$-curves. Suppose that $\gamma_1,\cdots,\gamma_m\in H^*(X,\bQ)(m\geqslant0)$, and $g\in\bZ$. Then we have
\ba
&&n_{g,\beta}^X(\gamma_1,\cdots,\gamma_m)=n_{g,\cF\beta}^{X'}(\cF\gamma^1,\cdots,\cF\gamma_m),\quad\forall\beta\in H_2(X,\bZ)\setminus Cen(f);\label{BPSflop}\\
&&n_{g,\beta}^X=n_{g,-\cF\beta}^{X'},\quad\forall\beta\in Cen(f)\setminus\{0\}.\label{BPScenter}
\na
\end{corollary}
The Donaldson-Thomas theory of X counts embedded curves on X only in a virtual sense. A fundamental problem in the Donaldson-Thomas theory is to understand the hidden enumerative meanings of the invariants. It is conjectured that BPS state counts are enumerative. It is interesting to understand Corollary \ref{BPS} from the point of view of enumerative geometry. 

As another application, we investigate the conjectural GW/DT correspondence. In the primary case, the correspondence is established for several classes of $3$-folds, including toric $3$-folds \cite{MOOP}, and Calabi-Yau $3$-folds which are complete intersections in products of projective spaces \cite{PP,T1}. However, in the descendent case, not much is known. The following result gives further positive evidence to the MNOP conjecture.

\begin{corollary}\label{GWDT}
Let $f$ be a flop of a disjoint union of $(-2)$-curves. Assume that $X$ satisfies the GW/DT correspondence for primary insertions and descendants of the point class. Then so does $X'$.
\end{corollary}

We observe that Toda's flop formulae (Theorem 1.2 in \cite{T2}) are analogous to Proposition \ref{mainresult}, and we can check that Corollary \ref{BPS} and \ref{GWDT} also hold for general flops between Calabi-Yau $3$-folds. Based on the established flop formulae of \cite{HL}, \cite{T2} and ours, we propose the following conjecture.

\begin{conjecture}\label{conj}
The formulae \eqref{flop} and \eqref{center} hold for general flops.
\end{conjecture}

We expect that the degeneration formula will play a role in the proof of the conjecture. Note that an embedded nonsingular rational curve in a $3$-fold is locally floppable only if it has normal bundle of type $(-1,-1),(0,-2)$ or $(1,-3)$ \cite{La}. However, unlike the case of $(-2)$-curves, it is difficult to describe a general flop of $(1,-3)$-curves in terms of blow-ups and blow-downs (see \cite{Pi} for some explicit examples).

Most of the results obtained in this paper also hold in the stable pair theory \cite{PT}, since the behavior of stable pair invariants under degeneration is similar to that of Donaldson-Thomas theory. We also have corresponding corollaries on BPS state counts and GW/P correspondence, and conjectural flop formulae for general flops in the stable pair theory.

An outline of this paper is as follows. In Section 2, we review some basic materials on Donaldson-Thomas invariants. In Section 3, we recall Reid's result  to decompose the flop under consideration into a sequence of blow-ups followed by a sequence of blow-downs, and use the degeneration formula to prove Proposition \ref{mainresult}. In Section 4, we give a working definition of the BPS state counts for Donaldson-Thomas theory and prove Corollary \ref{BPS}. In Section 5, we review the conjectural GW/DT corespondence and prove Corollary \ref{GWDT}.

\section{Preliminaries on Donaldson-Thomas invariants}
In this section, we briefly review some basic materials on Donaldson-Thomas invariants and fix notations.  We refer readers to \cite{DT, LW, MNOP1,MNOP2, Th} for details.

Donaldson-Thomas theory is defined via integration over the moduli space of ideal sheaves of $X$. Here an ideal sheaf is a torsion-free sheaf of rank $1$ with trivial determinant. Each ideal sheaf $\sI$ determines a subscheme $Y\subset X$ via the exact sequence
\ban
0\rightarrow\sI\rightarrow\cO_X\rightarrow\cO_Y\rightarrow 0.
\nan
In this note, we will consider only the case dim$Y\leqslant1$. The one dimensional components of $Y$ (weighted by their intrinsic multiplicities) determine an element,
\ban
[Y]\in H_2(X,\bZ).
\nan
For $n\in\bZ$ and $\beta\in H_2(X,\bZ)$, let $I_n(X,\beta)$ be the moduli space of ideal sheaves $\sI$ satisfying 
\ban
\chi(\cO_Y)=n,\quad [Y]=\beta,
\nan
where $\chi$ is the holomorphic Euler characteristic. From the deformation theory, $I_n(X,\beta)$ carries a virtual fundamental class of degree $\int_\beta c_1(X)$.

For $d\in\bZ_{\geqslant 0}$ and $\gamma\in H^*(X,\bQ)$, the descendant insertion $\tilde\tau_d(\gamma)$ is defined as follows. Let 
\ban
\pi_X:X\times I_n(X,\beta)&\rightarrow&X,\\
\pi_I:X\times I_n(X,\beta)&\rightarrow&I_n(X,\beta)
\nan
be tautological projections. Let $\cI$ be the universal sheaf over $X\times I_n(X,\beta)$. The operation
\ban
(-1)^{d+1}\pi_{I*}\bigg(\pi_X^*(\gamma)\cdot\textrm{ch}_{2+d}(\cI)\cap\pi_I^*(\cdot)\bigg):H_*(I_n(X,\beta),\bQ)\rightarrow H_*(I_n(X,\beta),\bQ)
\nan
is the action of $\tilde\tau_d(\gamma)$. The Donaldson-Thomas invariants with descendant insertions are defined as the virtual integration
\ban
\<\prod\limits_{i=1}^m\tilde\tau_{d_i}(\gamma_i)\>_{n,\beta}&=&\int_{[I_n(X,\beta)]^{vir}}\prod\limits_{i=1}^m\tilde\tau_{d_i}(\gamma_i),
\nan
where $d_1\cdots,d_m\in\bZ_{\geqslant 0}$, and $\gamma_1,\cdots,\gamma_m\in H^*(X,\bQ)$. Here the integral is the push-forward to a point of the class
\ban
\tilde\tau_{d_i}(\gamma_i)\circ\cdots\tilde\tau_{d_m}(\gamma_m)([I_n(X,\beta)]^{vir}).
\nan
The partition function of the Donaldson-Thomas invariants is defined by
\ban
Z_{DT}\bigg(X;q|\prod\limits_{i=1}^m\tilde\tau_{d_i}(\gamma_i)\bigg)_\beta=\sum\limits_{n\in\bZ}\<\prod\limits_{i=1}^m\tilde\tau_{d_i}(\gamma_i)\>_{n,\beta}q^n,
\nan
and the reduced partition function is obtained by formally removing the degree zero contributions,
\ban
Z'_{DT}\bigg(X;q|\prod\limits_{i=1}^m\tilde\tau_{d_i}(\gamma_i)\bigg)_\beta=\frac{Z_{DT}\bigg(X;q|\prod\limits_{i=1}^m\tilde\tau_{d_i}(\gamma_i)\bigg)_\beta}{Z_{DT}\bigg(X;q|\bigg)_0}.
\nan

Let $S\subset X$ be a nonsingular divisor. For $n\in\bZ$ and nonzero $\beta\in H_2(X,\bZ)$ with $\int_\beta[S]\geqslant 0$, let $I_n(X/S,\beta)$ be the moduli space of relative ideal sheaves, which carries a virtual fundamental class of degree $\int_\beta c_1(X)$. We have the following natural morphism
\ban
\epsilon:I_n(X/S,\beta)\rightarrow\textrm{Hilb}(S,\int_\beta[S])
\nan
The pull-back of cohomology classes of Hilb$(S,\int_\beta[S])$ gives relative insertions. 

Let us briefly recall Nakajima basis for the cohomology of Hilbert schemes of points of $S$. Let $\{\delta_i\}$ be a basis of $H^*(S,\bQ)$ with dual basis $\{\delta^i\}$. For any cohomology weighted partition $\eta$ with respect to the basis $\{\delta_i\}$, Nakajima constructed a cohomology class $C_\eta\in H^*(\textrm{Hilb}(S,|\eta|),\bQ)$. The Nakajima basis of  $H^*(\textrm{Hilb}(S,d),\bQ)$ is the set $\{C_\eta\}_{|\eta|=d}$. We refer readers to \cite{Na} for more details.

The partition function of the relative Donaldson-Thomas invariants is defined by
\ban
Z_{DT}\bigg(X/S;q|\prod\limits_{i=1}^m\tilde\tau_{d_i}(\gamma_i)|\eta\bigg)_\beta=\sum\limits_{n\in\bZ}q^n\int_{[I_n(X/S,\beta)]^{vir}}\prod\limits_{i=1}^m\tilde\tau_{d_i}(\gamma_i)\cdot\epsilon^*C_\eta,
\nan
and the reduced partition function is obtained by formally removing the degree zero contributions
\ban
Z'_{DT}\bigg(X/S;q|\prod\limits_{i=1}^m\tilde\tau_{d_i}(\gamma_i)|\eta\bigg)_\beta=\frac{Z_{DT}\bigg(X/S;q|\prod\limits_{i=1}^m\tilde\tau_{d_i}(\gamma_i)|\eta\bigg)_\beta}{Z_{DT}\bigg(X/S;q||\bigg)_0}
\nan 

Let $\Delta\subset\bC$ be the unit disc, and let $\pi : \chi \rightarrow\Delta$ be a nonsingular $4$-fold over $\bD$, such that $\chi_t = \pi^{-1}(t) \cong X $ for $t\not= 0$, and $\chi_0$ is a union of two irreducible nonsingular projective $3$-folds $X_1$
and $X_2$ intersecting transversally along a nonsingular projective surface $S$. (We can also consider the general case where the central fiber has several irreducible components, but we restrict ourselves to this simple case for simplicity of presentation.)  Consider the natural inclusion maps
$$
  i_t: X=\chi_t \longrightarrow \chi,\,\,\,\,\,\,\,\,
  i_0:\chi_0\longrightarrow \chi,
$$
and the gluing map
$$
  g= (j_1,j_2) : X_1\coprod X_2\longrightarrow \chi_0.
$$
We have
$$
  H_2(X,\bZ)\stackrel{i_{t*}}{\longrightarrow}
  H_2(\chi,\bZ)\stackrel{i_{0_*}}{\longleftarrow}
  H_2(\chi_0,\bZ)\stackrel{g_*}{\longleftarrow} H_2(X_1,\bZ)\oplus
  H_2(X_2,\bZ),
$$
where $i_{0*}$ is an isomorphism since there exists a deformation retract from $\chi$ to $\chi_0$(see \cite{Cl}). Also, since the family $\chi\rightarrow\bA^1$ comes from a trivial family, it follows that each $\gamma\in H^*(X,\bQ) $ has
global liftings such that the restriction $\gamma(t)$ on $\chi_t$ is defined for all $t$.

The degeneration formula for the Donaldson-Thomas theory expresses the absolute invariants of $X$ via the relative invariants of $(X_1,S)$ and $(X_2,S)$:
\ban
&&Z'_{DT}\bigg(X;q|\prod\limits_{i=1}^m\tilde\tau_{d_i}(\gamma_i)\bigg)_\beta\\
&=&\sum Z'_{DT}\bigg(X_1/S;q|\prod\limits_{i\in P_1}\tilde\tau_{d_i}(j_1^*\gamma_i(0))|\eta\bigg)_{\beta_1}\cdot\frac{(-1)^{|\eta|-\ell(\eta)}\fz(\eta)}{q^{|\eta|}}\cdot Z'_{DT}\bigg(X_2/S;q|\prod\limits_{i\in P_2}\tilde\tau_{d_i}(j_2^*\gamma_i(0))|\eta^\vee\bigg)_{\beta_2},
\nan
where $\fz(\eta)=|\textrm{Aut}(\eta)|\cdot\prod\limits_{i=1}^{\ell(\eta)}\eta_i$, $\eta^\vee$ is defined by taking the Poincar\'e duals of the cohomology weights of $\eta$, and the sum is over cohomology weighted partitions $\eta$, degree splittings $i_{t*}\beta=i_{0*}(j_{1*}\beta_1+j_{2*}\beta_2)$, and marking partitions $P_1\coprod P_2=\{1,\cdots,m\}$. In particular, if $(\eta,\beta_1,\beta_2)$ has nontrivial contribution in the degeneration formula, then we have the following dimension constraint:
\ban
v\textrm{dim}_\bC P_n(X_1/S,\beta_1)+v\textrm{dim}_\bC P_n(X_2/S,\beta_2)=v\textrm{dim}_\bC P_n(X,\beta)+2|\eta|.
\nan

\section{Proof of main result}

In this section, we give a detailed proof o Proposition \ref{mainresult}. We first recall Reid's result to decompose a flop of a disjoint union of  $(-2)$-curves into a sequence of blow-ups followed by a sequence of blow-downs, and then use the degeneration formula to prove our flop formula. We refer readers to \cite{Re} for explicit local description of the flop of a single $(-2)$-curve, and to \cite{Ko, KM} for general materials on birational geometry of $3$-folds.

Let $C_{1},\cdots,C_{l}$ be the irreducible components of the center of $f$. We can contract these curves to obtain a contraction $\psi:X\rightarrow\bar X$, and then these curves generate an extremal face in $NE(X)$. The width of $C_{i}$ in $X$ is defined by Reid as follows \cite{Re}:
\ban
w_{i}&:=&width(C_{i}\subset X)\\
&:=&\sup\{k|\textrm{ there exists  a scheme }S\cong C_{i}\times\textrm{Spec}(\bC[\epsilon]/\epsilon^k)\textrm{ such that }C_{i}\subset S\subset X\}.
\nan
Since $C_{i}$ is isolated, it follows that $1\leqslant w_{i}<\infty$. Note that $\psi(C_{i})\in\bar X$ is a hypersurface singularity given by
\ban
x^2+y^2+z^2+t^{2w_{i}}=0.
\nan
In particular, $C_{i}$ is a $(-1,-1)$-curve if and only if $w_{i}=1$.

Without loss of generality, assume that 
\ban
w_{1}\geqslant\cdots\geqslant w_{l}\geqslant 1.
\nan
Let $w=w_{1}$, and for $d=1,\cdots,w$, set
\ban
k_d:=\sup\{i|w_{i}\geqslant d\}.
\nan
Then 
\ban
1\leqslant k_w\leqslant\cdots\leqslant k_1=l.
\nan

Write $X=X_0$ and $C_i=C_{0,i}$. Then proceeding inductively, we obtain a sequence of blow-ups:
\ban
X_w\xrightarrow{\phi_{w-1}}X_{w-1}\xrightarrow{\phi_{w-2}}\cdots\xrightarrow{\phi_1}X_1\xrightarrow{\phi_0}X_0. 
\nan
Here for $d=0,1,\cdots,w-2$, $\phi_d$ is the blow-up of $X_{d}$ along the $(-2)$-curves $C_{d,1},\cdots,C_{d,k_{d+1}}$. Let
\ban
E_{d+1,i}:=\phi_d^{-1}(C_{d,i})\cong\left\{\begin{array}{cl}\bF_2,&i=1,\cdots,k_{d+2},\\\bF_0,&i=k_{d+1}+1,\cdots,k_{d+1}.\end{array}\right.
\nan
For $i=1,\cdots,k_{d+2}$, $C_{d+1,i}\subset E_{d+1,i}$ is the unique nonsingular rational curve with negative self intersection number, which is also a $(-2)$-curve in $X_{d+1}$ with
\ban
width(C_{d+1,i}\subset X_{d+1})=w_{i}-d-1,\quad d=1,\cdots,w-2.
\nan
Moreover, $\phi_{w-1}$ is the blow-up of $X_{w-1}$ along the $(-1,-1)$-curves $C_{w-1,1},\cdots,C_{w-1,k_{w}}$, and
\ban
E_{w,i}:=\phi_{w-1}^{-1}(C_{w-1,i})\cong\bF_0,\quad i=1,\cdots,k_w.
\nan

For $d=1,\cdots,w-1$ and $i=1,\cdots,k_{d+1}$, the strict transform of $E_{d,i}$ under $\phi_d$, denoted by $\tilde E_{d,i}$, is isomorphic to $E_{d,i}$. Moreover, $\tilde E_{d,i}\cap E_{d+1,i}$ is a nonsingular rational curve, which has negative self intersection number on $\tilde E_{d,i}$, and self intersection number $2$ on $E_{d+1,i}$. In particular, $\tilde E_{w-1,i}\cap E_{w,i}$ is a $(1,1)$-curve on $E_{w,i}\cong\bF_0$. Note that $\tilde E_{d,i}$ is not affected by blow-ups $\phi_{d+1},\cdots,\phi_{w-1}$, and can be viewed as an embedded surface in $X_w$. For $d=1,\cdots,w-1$ and $i=k_{d+1}+1,\cdots,k_{d}$, $E_{d,i}$ is not affected by blow-ups $\phi_d,\cdots,\phi_{w-1}$, and can be viewed as an embedded surface in $X_w$.

Write $E_{w,i}=E'_{w,i}$. Since each $E'_{w,i}\cong\bF_0$ has a ruling not contracted by $\phi_{w-1}$, it follows that we can blow down $X_w$ along these rulings for all $i$ simultaneously to obtain $\phi'_{w-1}:X_w\rightarrow X'_{w-1}$. 
Proceeding inductively, we also have a sequence of blow-downs:
\ban
X_w\xrightarrow{\phi'_{w-1}}X'_{w-1}\xrightarrow{\phi'_{w-2}}\cdots\xrightarrow{\phi'_1}X'_1\xrightarrow{\phi'_0}X'_0. 
\nan
For $d=0,1,\cdots,w-2$, let
\ban
C'_{w-1-d,i}:=\phi'_{w-1-d}(E'_{w-d,i}),\quad i=1,\cdots,k_{w-d}
\nan
and 
\ban
E'_{w-1-d,i}=\left\{\begin{array}{cl}\phi'_{w-1-d}\circ\cdots\circ\phi'_{w-1}(\tilde E_{w-1-d,i})\cong\bF_2,&i=1,\cdots,k_{w-d},\\\phi'_{w-1-d}\circ\cdots\circ\phi'_{w-1}(E_{w-1-d,i})\cong\bF_0&i=k_{w-d}+1,\cdots,k_{w-1-d}.\end{array}\right.
\nan
Then $C'_{w-1-d,i}$ is a $(-2)$-curve in $X'_{w-1-d}$ with
\ban
width(C'_{w-1-d,i}\subset X'_{w-1-d})=w_i-w+1+d,
\nan
and $C'_{w-1-d,i}\subset E'_{w-1-d,i}$ is the unique nonsingular rational curve with negative self intersection number.  Since for $i=1,\cdots,k_{w-d}$, each $E'_{w-1-d,i}\cong\bP_{C'_{w-1-d,i}}(\cO\oplus\cO(-2))$ has a fiber ruling, and for $i=k_{w-d}+1,\cdots,k_{w-1-d}$, each $E_{w-1-d,i}$ has a ruling not contracted by $\phi_{w-1-d}$, it follows that we can blow down $X'_{w-1-d}$ along these rulings simultaneously to obtain $\phi'_{w-2-d}:X'_{w-1-d}\rightarrow X'_{w-2-d}$.

Now for $d=0,1,\cdots,w-1$, the birational map
\ban
f_d:=\phi'_d\circ\cdots\circ\phi'_{w-1}\circ\phi^{-1}_{w-1}\circ\cdots\circ\phi^{-1}_d:X_d\dashrightarrow X'_d
\nan 
is a flop of $(-2)$-curves $C_{d,1},\cdots,C_{d,k_{d+1}}$, where $C_{d,i}$ is flopped to $C'_{d,i}$. In particular, we have $X'=X'_0$ and $f=f_0$.

Degenerate $X$ along $C_1,\cdots,C_l$ simultaneously, and we have
\ban
&&Z'_{DT}\bigg(X;q|\prod\limits_{i=1}^m\tilde\tau_{d_i}(\gamma_i)\bigg)_\beta\\
&=&\sum Z'_{DT}\bigg(X_1/E_1;q|\prod\limits_{i=1}^m\tilde\tau_{d_i}(\phi_0^*\gamma_i)|\eta_1^\vee,\cdots,\eta_l^\vee\bigg)_{\tilde\beta}\prodl_{i=1}^l\frac{(-1)^{|\eta_i|-\ell(\eta_i)}\fz(\eta_i)}{q^{|\eta_i|}}Z'_{DT}(\bP_i/D_i;q||\eta_i)_{\beta_i},
\nan
where we have assumed that the support of $\gamma_i$ is away from $\bigcup\limits_{i=1}^lC_i$, and 
\ban
E_1&:=&\bigcup\limits_{i=1}^lE_{1,i},\\
\bP_i&:=&\bP_{C_i}(N_{C_i}\oplus\cO_{C_i}),\quad (N_{C_i}\textrm{ is the normal bundle of }C_i\textrm{ in }X)\\
D_i&:=&\bP_{C_i}(N_{C_i}\oplus\{0\}).
\nan
By dimension constraint, we find that $\eta_1=\cdots=\eta_l=\emptyset$. So 
\ban
\tilde\beta\cdot E_1=\beta_i\cdot D_i=0.
\nan
For $\tilde\beta$, note that $\phi_0$ induces a natural injection via 'pull-back' of $2$-cycles 
\ban
\phi_1^!=PD_{X_1}\circ\phi_0^*\circ PD_X:H_2(X,\bZ)\rightarrow H_2(X_1,\bZ),
\nan
where the image of $\phi_0^!$ is the subset of $H_2(X_1,\bZ)$ consisting of $2$-cycles having intersection number zero with $E_1$, and so we have
$\tilde\beta\in Im\phi_0^!$. For $\beta_i$, note that
\ban
H_2(\bP_i,\bZ)=\bZ[C_i]\oplus\bZ f_i,
\nan
where we have used the identification $C_i\cong\bP_{C_i}(\{0\}\oplus\cO_{C_i})$, and $f_i$ is the class of a line in the fiber of $\bP_i$. So $\beta_i\cdot D_i=0$ implies that $\beta_i\in\bZ_{\geqslant0}[C_i]$, since $\beta_i$ is effective. Therefore
\ban
&&Z'_{DT}\bigg(X;q|\prod\limits_{i=1}^m\tilde\tau_{d_i}(\gamma_i)\bigg)_\beta\\
&=&\suml_{\substack{\beta'\in H_2(X,\bZ),n_i\in\bZ_{\geqslant 0}\\\beta'+n_1[C_1]+\cdots n_l[C_l]=\beta}} Z'_{DT}\bigg(X_1/E_1;q|\prod\limits_{i=1}^m\tilde\tau_{d_i}(\phi_0^*\gamma_i)|\bigg)_{\phi_0^!\beta'}\prodl_{i=1}^lZ'_{DT}(\bP_i/D_i;q||)_{n_i[C_i]}.
\nan
In particular, since the irreducible curves in the center of $f$ generate an extremal face in $NE(X)$, it follows that for $\beta\in Cen(f)$, we have
\ban
Z'_{DT}\bigg(X;q|\bigg)_{\beta}=\suml_{\substack{\beta'+\suml_{i=1}^ln_i[C_i]=\beta\\\beta'\in Cen(f)}} Z'_{DT}\bigg(X_1/E_1;q||\bigg)_{\phi_0^!\beta'}\prodl_{i=1}^lZ'_{DT}(\bP_i/D_i;q||)_{n_i[C_i]}.
\nan
Therefore we have obtained the following:
\ba
\suml_{\beta\in H_2(X,\bZ)}v^\beta Z'_{DT}\bigg(X;q|\prod\limits_{i=1}^m\tilde\tau_{d_i}(\gamma_i)\bigg)_\beta&=&\suml_{\beta\in H_2(X,\bZ)}v^\beta Z'_{DT}\bigg(X_1/E_1;q|\prod\limits_{i=1}^m\tilde\tau_{d_i}(\phi_0^*\gamma_i)|\bigg)_{\phi_0^!\beta}\nonumber\\
&&\qquad\cdot\prodl_{i=1}^l\suml_{d\geqslant 0}v^{d[C_i]}Z'_{DT}(\bP_i/D_i;q||)_{d[C_i]},\nonumber\\
\suml_{\beta\in Cen(f)}v^{\beta}Z'_{DT}(X;q|)_{\beta}&=&\suml_{\beta\in Cen(f)}v^{\beta}Z'_{DT}\bigg(X_1/E_1;q||\bigg)_{\phi_0^!\beta}\cdot\prodl_{i=1}^l\suml_{d\geqslant 0}v^{d[C_i]}Z'_{DT}(\bP_i/D_i;q||)_{d[C_i]},\label{Xcenter}
\na
which implies that
\ba
\frac{\suml_{\beta\in H_2(X,\bZ)}v^\beta Z'_{DT}\bigg(X;q|\prod\limits_{i=1}^m\tilde\tau_{d_i}(\gamma_i)\bigg)_\beta}{\suml_{\beta\in Cen(f)}v^{\beta}Z'_{DT}(X;q|)_{\beta}}&=&\frac{\suml_{\beta\in H_2(X,\bZ)}v^\beta Z'_{DT}\bigg(X_1/E_1;q|\prod\limits_{i=1}^m\tilde\tau_{d_i}(\phi_0^*\gamma_i)|\bigg)_{\phi_0^!\beta}}{\suml_{\beta\in Cen(f)}v^{\beta}Z'_{DT}\bigg(X_1/E_1;q||\bigg)_{\phi_0^!\beta}}.\label{Xrel}
\na

Now degenerate $X_1$ along $E_{1,1},\cdots,E_{1,l}$ simultaneously, and we obtain
\ban
&&Z'_{DT}\bigg(X_1;q|\prod\limits_{i=1}^m\tilde\tau_{d_i}(\phi_0^*\gamma_i)\bigg)_{\phi_1^!\beta}\\
&=&\sum Z'_{DT}\bigg(X_1/E_1;q|\prod\limits_{i=1}^m\tilde\tau_{d_i}(\phi_0^*\gamma_i)|\eta_1^\vee,\cdots,\eta_l^\vee\bigg)_{\tilde\beta}\prodl_{i=1}^l\frac{(-1)^{|\eta_i|-\ell(\eta_i)}\fz(\eta_i)}{q^{|\eta_i|}}Z'_{DT}(\bP_{1,i}/D_{1,i};q||\eta_i)_{\beta_i},
\nan
where
\ban
\bP_{1,i}&:=&\bP_{E_{1,i}}(N_{E_{1,i}}\oplus\cO_{E_{1,i}}),\quad (N_{E_{1,i}}\textrm{ is the normal bundle of }E_{1,i}\textrm{ in }X_1)\\
D_{1,i}&:=&\bP_{E_{1,i}}(N_{E_{1,i}}\oplus\{0\}).
\nan
By dimension constraint, we find that $\eta_1=\cdots=\eta_l=\emptyset$. So we have
\ban
&&Z'_{DT}\bigg(X_1;q|\prod\limits_{i=1}^m\tilde\tau_{d_i}(\phi_0^*\gamma_i)\bigg)_{\phi_1^!\beta}\\
&=&\suml_{\substack{\beta'\in H_2(X,\bZ),\beta_i\in H_2(\bP_{1,i},\bZ)\\\phi_0^!\beta'+(\pi_{1,1})_*\beta_1+\cdots+(\pi_{1,l})_*\beta_l=\phi_0^!\beta\\\beta_i\cdot E_{1,i}=\beta_i\cdot D_i=0}} Z'_{DT}\bigg(X_1/E_1;q|\prod\limits_{i=1}^m\tilde\tau_{d_i}(\phi_0^*\gamma_i)|\bigg)_{\phi_0^!\beta'}\prodl_{i=1}^{l}Z'_{DT}(\bP_{1,i}/D_{1,i};q||)_{\beta_i},
\nan
where we have used the identification $E_{1,i}\cong\bP_{E_{1,i}}(\{0\}\oplus\cO_{E_{1,i}})$, and $\pi_{1,i}$ is the composition
\ban
\bP_{1,i}\rightarrow E_{1,i}\hookrightarrow X_1.
\nan
In particular, since
\ban
\beta'+(\phi_0)_*(\pi_{1,1})_*\beta_1+\cdots+(\phi_0)_*(\pi_{1,l})_*\beta_l=\beta,
\nan
it follows that for $\beta\in Cen(f)$, we have
\ban
&&Z'_{DT}\bigg(X_1;q|\bigg)_{\phi_0^!\beta}\\
&=&\suml_{\substack{\beta'+(\phi_0)_*(\pi_{1,1})_*\beta_1+\cdots+(\phi_0)_*(\pi_{1,l})_*\beta_l=\beta\\\beta'\in Cen(f)\\\beta_i\cdot E_{1,i}=\beta_i\cdot D_{1,i}=0}} Z'_{DT}\bigg(X_1/E_1;q||\bigg)_{\phi_0^!\beta'}\prodl_{i=1}^lZ'_{DT}(\bP_{1,i}/D_{1,i};q||)_{\beta_i},
\nan
So we have obtained
\ba
\suml_{\beta\in H_2(X,\bZ)}v^{\beta} Z'_{DT}\bigg(X_1;q|\prod\limits_{i=1}^m\tilde\tau_{d_i}(\phi_0^*\gamma_i)\bigg)_{\phi_0^!\beta}&=&\suml_{\beta\in H_2(X,\bZ)}v^{\beta} Z'_{DT}\bigg(X_1/E_1;q|\prod\limits_{i=1}^m\tilde\tau_{d_i}(\phi_0^*\gamma_i)|\bigg)_{\phi_0^!\beta}\nonumber\\
&&\qquad\cdot\prodl_{i=1}^{l}\suml_{d\geqslant0}v^{d[C_i]}\suml_{\substack{\beta_i\in H_2(\bP_{1,i},\bZ)\\\beta_i\cdot E_{1,i}=\beta_i\cdot D_{1,i}=0\\(\phi_0)_*(\pi_{1,i})_*\beta_i=d[C_i]}}Z'_{DT}(\bP_{1,i}/D_{1,i};q||)_{\beta_i},\nonumber\\
\suml_{\beta\in Cen(f)}v^{\beta}Z'_{DT}(X_1;q|)_{\phi_0^!\beta}&=&\suml_{\beta\in Cen(f)}v^{\beta}Z'_{DT}\bigg(X_1/E_1;q||\bigg)_{\phi_0^!\beta}\label{X_1center}\\
&&\qquad\cdot\prodl_{i=1}^{l}\suml_{d\geqslant0}v^{d[C_i]}\suml_{\substack{\beta\in H_2(\bP_{1,i},\bZ)\\\beta\cdot E_{1,i}=\beta\cdot D_{1,i}=0\\(\phi_0)_*(\pi_{1,i})_*\beta=d[C_i]}}Z'_{DT}(\bP_{1,i}/D_{1,i};q||)_{\beta},\nonumber
\na
which implies that 
\ba
&&\frac{\suml_{\beta\in H_2(X,\bZ)}v^{\beta} Z'_{DT}\bigg(X_1;q|\prod\limits_{i=1}^m\tilde\tau_{d_i}(\phi_0^*\gamma_i)\bigg)_{\phi_0^!\beta}}{\suml_{\beta\in Cen(f)}v^{\beta}Z'_{DT}(X_1;q|)_{\phi_0^!\beta}}\nonumber\\
&=&\frac{\suml_{\beta\in H_2(X,\bZ)}v^{\beta} Z'_{DT}\bigg(X_1/E_1;q|\prod\limits_{i=1}^m\tilde\tau_{d_i}(\phi_0^*\gamma_i)|\bigg)_{\phi_0^!\beta}}{\suml_{\beta\in Cen(f)}v^{\beta}Z'_{DT}\bigg(X_1/E_1;q||\bigg)_{\phi_0^!\beta}}.\label{X_1rel}
\na
Then from \eqref{Xrel} and \eqref{X_1rel}, we have
\ba
\frac{\suml_{\beta\in H_2(X,\bZ)}v^\beta Z'_{DT}\bigg(X;q|\prod\limits_{i=1}^m\tilde\tau_{d_i}(\gamma_i)\bigg)_\beta}{\suml_{\beta\in Cen(f)}v^{\beta}Z'_{DT}(X;q|)_{\beta}}&=&\frac{\suml_{\beta\in H_2(X,\bZ)}v^{\beta} Z'_{DT}\bigg(X_1;q|\prod\limits_{i=1}^m\tilde\tau_{d_i}(\phi_0^*\gamma_i)\bigg)_{\phi_0^!\beta}}{\suml_{\beta\in Cen(f)}v^{\beta}Z'_{DT}(X_1;q|)_{\phi_0^!\beta}}.\label{X}
\na
Using the identification $\cF:H_2(X,\bZ)\xrightarrow{\cong}H_2(X',\bZ)$, we also have
\ba
&&\frac{\suml_{\beta\in H_2(X,\bZ)}v^\beta Z'_{DT}\bigg(X';q|\prod\limits_{i=1}^m\tilde\tau_{d_i}(\cF\gamma_i)\bigg)_{\cF\beta}}{\suml_{\beta\in Cen(f)}v^{\beta}Z'_{DT}(X';q|)_{\cF\beta}}\nonumber\\
&=&\frac{\suml_{\beta\in H_2(X,\bZ)}v^{\beta} Z'_{DT}\bigg(X'_1;q|\prod\limits_{i=1}^m\tilde\tau_{d_i}(((\phi'_0)^*\cF\gamma_i)\bigg)_{((\phi'_0)^!\cF\beta}}{\suml_{\beta\in Cen(f)}v^{\beta}Z'_{DT}(X'_1;q|)_{(\phi'_0)^!\cF\beta}}.\label{X'}
\na

Now we use induction on $w=1,2,3,\cdots$ to prove \eqref{flop} in Proposition \ref{mainresult}. For $w=1$, we have the following observation.
\begin{lemma}\label{1stlemma}
For any nonzero $\beta\in Cen(f)$, $\phi_0^!\beta$ is not effective. 
\end{lemma}
\begin{proof}
Argue by contradiction, and then $\beta=(\phi_0)_*\phi_0^!\beta$ is also effective. We can write $\beta=\suml_{i=1}^la_i[C_i]$ with $a_i\in\bZ_{\geqslant 0}$. Note that $\cF[C_i]=-[C'_i]$, and then
\ban
(\phi'_0)_*\phi_0^!\beta=\cF\beta=-\suml_{i=1}^la_i[C'_i].
\nan
Since $\suml_{i=1}^la_i[C'_i]$ is effective, it follows that $(\phi'_0)_*\phi_0^!\beta$ is not effective, which implies that $\phi_0^!\beta$ is not effective.
\end{proof}
Therefore, 
\ban
\suml_{\beta\in Cen(f)}v^{\beta}Z'_{DT}(X_1;q|)_{\phi_0^!\beta}=1\textrm{ and }\suml_{\beta'\in Cen(f^{-1})}v^{\beta'}Z'_{DT}(X_1;q|)_{\phi_0^!\beta'}=1.
\nan
Note that in \eqref{X} and \eqref{X'}, we have
\ban
\phi_0^*\gamma_i=(\phi'_0)^*\cF\gamma_i\textrm{ and }\phi_0^!\beta=(\phi'_0)^!\cF\beta.
\nan
So in the case $w=1$, \eqref{flop} follows from \eqref{X} and \eqref{X'}.

Assume that the case for $w=W\geqslant1$ is proved. Then for $w=W+1$, we have
\ban
\frac{\suml_{\beta_1\in H_2(X_1,\bZ)}v^{\beta_1} Z'_{DT}(X_1;q|\prodl_{i=1}^m\tilde\tau_{d_i}(\phi_0^*\gamma_i))_{\beta_1}}{\suml_{\beta_1\in Cen(f_1)}v^{\beta_1} Z'_{DT}(X_1;q|)_{\beta_1}}&=&\frac{\suml_{\beta_1\in H_2(X_1,\bZ)}v^{\beta_1} Z'_{DT}(X'_1;q|\prodl_{i=1}^m\tilde\tau_{d_i}(\cF_1\phi_0^*\gamma_i))_{\cF_1\beta_1}}{\suml_{\beta_1\in Cen(f_1)}v^{\beta_1} Z'_{DT}(X'_1;q|)_{\cF_1\beta_1}},
\nan
where $\cF_1$ is the correspondence on (co)homology groups induced by $f_1$. We have the following key observation.

\begin{lemma}\label{2ndlemma}
Let $S=$Span$_\bZ\{[C_1],\cdots,[C_{k_2}]\}$. For any $\beta\in Cen(f)\setminus S$, $\phi_0^!\beta$ is not effective.
\end{lemma}
\begin{proof}
Without loss of generality, assume that 
\ban
S\cap\{[C_{k_2+1}],\cdots,[C_l]\}=\{[C_{l'}],\cdots,[C_l]\}.
\nan
Argue by contradiction, and we can write $\phi_0^!\beta=\suml_{j=1}^nm_j[V_j]$, where $m_j\in\bZ_{\geqslant0}$, and $V_1,\cdots,V_n$ are mutually distinct irreducible curves in $X_1$. Since 
\ban
\beta=(\phi_0)_*\phi_0^!\beta\in Cen(f)\setminus S,
\nan
and $[C_1],\cdots,[C_l]$ generate an extremal face in $NE(X)$, it follows that, for each $j$, $V_j$ is mapped onto a point or some $C_i$. In the former case, $V_j$ is a fiber of one irreducible component of $E_1$ and then $V_j\cdot E_1<0$. In the latter case, $V_j$ is contained in $E_{1,i}$ and then $V_j\cdot E_{1,i}\leqslant 0$. Moreover, we can find some $V_j$ which is contained in some $E_{1,i}\cong\bF_0$ for $l'\leqslant i\leqslant l$, and then $V_j\cdot E_{1,i}<0$. In sum, we have $\phi_0^!\beta\cdot E_1<0$, which is absurd.
\end{proof}
Since $Cen(f_1)=\{\phi_0^!\beta:\beta\in S\}$, it follows that 
\ban
\suml_{\beta_1\in Cen(f_1)}v^{\beta_1} Z'_{DT}(X_1;q|)_{\beta_1}=\suml_{\beta\in Cen(f)}v^{\phi_0^!\beta}Z'_{DT}(X_1;q|)_{\phi_0^!\beta}.
\nan
Now we have
\ban
&&\suml_{\beta_1\in Cen(f_1)}v^{\cF_1\beta_1}Z'_{DT}(X'_1;q|)_{\cF_1\beta_1}\\
&=&\suml_{\beta'_1\in Cen(f_1^{-1})}v^{\beta'_1}Z'_{DT}(X'_1;q| )_{\beta'_1}\\
&=&\suml_{\beta'\in Cen(f^{-1})}v^{(\phi'_0)^!\beta'}Z'_{DT}(X'_1;q|)_{(\phi'_0)^!\beta'}\\
&=&\suml_{\beta\in Cen(f)}v^{(\phi'_0)^!\cF\beta}Z'_{DT}(X'_1;q|)_{(\phi'_0)^!\cF\beta}\\
&=&\suml_{\beta\in Cen(f)}v^{\cF_1\phi_0^!\beta}Z'_{DT}(X'_1;q|)_{(\phi'_0)^!\cF\beta},
\nan
which implies that 
\ban
\frac{\suml_{\beta_1\in H_2(X_1,\bZ)}v^{\beta_1} Z'_{DT}(X_1;q|\prodl_{i=1}^m\tilde\tau_{d_i}(\phi_0^*\gamma_i))_{\beta_1}}{\suml_{\beta\in Cen(f)}v^{\phi_0^!\beta}Z'_{DT}(X_1;q|)_{\phi_0^!\beta}}&=&\frac{\suml_{\beta_1\in H_2(X_1,\bZ)}v^{\beta_1} Z'_{DT}(X'_1;q|\prodl_{i=1}^m\tilde\tau_{d_i}(\cF_1\phi_0^*\gamma_i))_{\cF_1\beta_1}}{\suml_{\beta\in Cen(f)}v^{\phi_0^!\beta}Z'_{DT}(X'_1;q|)_{(\phi'_0)^!\cF\beta}}.
\nan
Observe that we have the following decomposition
\ban
H_2(X_1,\bZ)=\phi_0^!H_2(X,\bZ)\oplus\bZ f_{1,1}\oplus\cdots\oplus\bZ f_{1,l},
\nan
where $f_{1,i}$ is the class of a fiber in $E_{1,i}$. So we obtain
\ba
&&\frac{\suml_{\beta\in H_2(X,\bZ)}v^{\beta} Z'_{DT}(X_1;q|\prodl_{i=1}^m\tilde\tau_{d_i}(\phi_0^*\gamma_i))_{\phi_0^!\beta}}{\suml_{\beta\in Cen(f)}v^{\beta}Z'_{DT}(X_1;q|)_{\phi_0^!\beta}}\nonumber\\
&=&\frac{\suml_{\beta\in H_2(X,\bZ)}v^{\beta} Z'_{DT}(X'_1;q|\prodl_{i=1}^m\tilde\tau_{d_i}(\cF_1\phi_0^*\gamma_i))_{\cF_1\phi_0^!\beta}}{\suml_{\beta\in Cen(f)}v^{\beta}Z'_{DT}(X'_1;q|)_{(\phi'_0)^!\cF\beta}}.\label{X_1X'_1}
\na
Note that 
\ban
\cF_1\phi_0^*\gamma_i=(\phi'_0)^!\cF\gamma_i\textrm{ and }\cF_1\phi_0^!\beta=(\phi'_0)^!\cF\beta,
\nan
and we see that in the case $w=W+1$, \eqref{flop} follows from \eqref{X}, \eqref{X'} and \eqref{X_1X'_1}.

To prove \eqref{center}, we have the following observation. Using the identification $-\cF:H_2(X,\bZ)\xrightarrow{\cong}H_2(X',\bZ)$, from \eqref{Xcenter}, we have
\ba
\suml_{\beta\in Cen(f)}v^{-\beta}Z'_{DT}(X';q|)_{\cF\beta}&=&\suml_{\beta\in Cen(f)}v^{-\beta}Z'_{DT}\bigg(X'_1/E'_1;q||\bigg)_{(\phi'_0)^!\cF\beta}\label{X'center}\\
&&\cdot\prodl_{i=1}^l\suml_{d\geqslant 0}v^{d[C_i]}Z'_{DT}(\bP'_i/D'_i;q||)_{d[C'_i]},\nonumber
\na
where
\ban
E'_1&:=&\bigcup\limits_{i=1}^lE'_{1,i},\\
\bP'_i&:=&\bP_{C'_i}(N_{C'_i}\oplus\cO_{C'_i}),\quad (N_{C'_i}\textrm{ is the normal bundle of }C'_i\textrm{ in }X')\\
D'_i&:=&\bP_{C'_i}(N_{C'_i}\oplus\{0\}),\\
C'_i&\cong&\bP_{C'_i}(\{0\}\oplus\cO_{C'_i}).
\nan
So from \eqref{Xcenter} and \eqref{X'center}, \eqref{center} is equivalent to the following:
\ba
\suml_{\beta\in Cen(f)}v^{\beta}Z'_{DT}\bigg(X_1/E_1;q||\bigg)_{\phi_0^!\beta}=\suml_{\beta\in Cen(f)}v^{-\beta}Z'_{DT}\bigg(X'_1/E'_1;q||\bigg)_{(\phi'_0)^!\cF\beta}\label{equivalenttocenter}
\na

Now we use induction on $w=1,2,3,\cdots$ to prove \eqref{center} (or \eqref{equivalenttocenter}). For $w=1$, Lemma \ref{1stlemma} implies that both LHS and RHS of \eqref{equivalenttocenter} are equal to $1$. Assume that the case for $w=W\geqslant1$ is proved. Then for $w=W+1$, we have
\ban
\suml_{\beta_1\in Cen(f_1)}v^{\beta_1}Z'_{DT}(X_1;q|)_{\beta_1}=\suml_{\beta_1\in Cen(f_1)}v^{-\beta_1}Z'_{DT}(X'_1;q|)_{\cF_1\beta_1},
\nan
and by Lemma \ref{2ndlemma}, this gives
\ba
\suml_{\beta\in Cen(f)}v^{\beta}Z'_{DT}(X_1;q|)_{\phi_0^!\beta}=\suml_{\beta\in Cen(f)}v^{-\beta}Z'_{DT}(X'_1;q|)_{\cF_1\phi_0^!\beta},\label{X_1X'_1center}
\na
Note that using the identification $-\cF:H_2(X,\bZ)\xrightarrow{\cong}H_2(X',\bZ)$, \eqref{X_1center} gives
\ba
\suml_{\beta\in Cen(f)}v^{-\beta}Z'_{DT}(X'_1;q|)_{(\phi'_0)^!\cF\beta}&=&\suml_{\beta\in Cen(f)}v^{-\beta}Z'_{DT}\bigg(X'_1/E'_1;q||\bigg)_{(\phi'_0)^!\cF\beta}\nonumber\\
&&\qquad\cdot\prodl_{i=1}^{l}\suml_{d\geqslant0}v^{d[C_i]}\suml_{\substack{\beta\in H_2(\bP'_{1,i},\bZ)\\\beta\cdot E'_{1,i}=\beta\cdot D'_{1,i}=0\\(\phi'_0)_*(\pi'_{1,i})_*\beta=d[C'_i]}}Z'_{DT}(\bP'_{1,i}/D'_{1,i};q||)_{\beta},\label{X'_1center}
\na
where
\ban
\bP'_{1,i}&:=&\bP_{E'_{1,i}}(N_{E'_{1,i}}\oplus\cO_{E'_{1,i}}),\quad (N_{E'_{1,i}}\textrm{ is the normal bundle of }E'_{1,i}\textrm{ in }X'_1)\\
D'_{1,i}&:=&\bP_{E'_{1,i}}(N_{E'_{1,i}}\oplus\{0\}),\\
E'_{1,i}&\cong&\bP_{E'_{1,i}}(\{0\}\oplus\cO_{E'_{1,i}}), 
\nan
and $\pi'_{1,i}$ is the composition $\bP'_{1,i}\rightarrow E'_{1,i}\hookrightarrow X'_1$. Note that 
\ban
\cF_1\phi_0^!\beta=(\phi'_0)^!\cF\beta,
\nan
and we see that in the case $w=W+1$, \eqref{equivalenttocenter} follows from \eqref{X_1center}, \eqref{X_1X'_1center} and \eqref{X'_1center}.

\section{BPS state counts}

BPS state counts were first introduced in the Gromov-Witten theory. In a study of Type IIA string theory via M-theory, Gopakumar and Vafa defined BPS state counts on Calabi-Yau $3$-folds \cite{GV1,GV2}. Motivated by the Calabi-Yau case together with the degenerate contribution computation, Pandharipande defined BPS state counts for arbitrary $3$-folds \cite{P1}. We refer interested readers to \cite{P2} for a precise description of the working definition of BPS state counts of Gromov-Witten theory of $X$. 

Now we give a  working definition of BPS state counts of Donaldson-Thomas theory. Let $\{T_i\}_{0\leqslant i\leqslant N}$ be a basis of $H^{*}(X,\bQ)$, and we define the BPS state counts of Donaldson-Thomas theory by the following identity:
\ban
&&\suml_{\substack{\beta\in H_2(X,\bZ)\\\int_\beta c_1(X)=0}}v^\beta Z'_{DT}(X;q|)_\beta+\suml_{\substack{\beta\in H_2(X,\bZ)\\\int_\beta c_1(X)>0}}v^\beta\suml_{e_0,\cdots,e_N\in\bZ_{\geqslant 0}}Z'_{DT}(X;q|\prodl_{i=0}^N\tilde\tau_0(T_i)^{e_i})_\beta\prodl_{i=0}^N\frac{t_i^{e_i}}{e_1!}\\
&=&\exp\bigg\{\suml_{\substack{\beta\in H_2(X,\bZ)\setminus\{0\}\\\int_\beta c_1(X)=0}}v^\beta\suml_{g\in\bZ}\suml_{r\in div(\beta)}n_{g,\frac\beta r}^X\cdot\frac{(-1)^{g-1}}{r}\Big[(-q)^r-2+(-q)^{-r}\Big]^{g-1}\\
&&\qquad+\suml_{\substack{\beta\in H_2(X,\bZ)\\\int_\beta c_1(X)>0}}v^\beta\suml_{g\in\bZ}\suml_{e_0,\cdots,e_N\in\bZ_{\geqslant 0}}n_{g,\beta}^X(\prodl_{i=0}^NT_i^{e_i})\prodl_{i=0}^N\frac{t_i^{e_i}}{e_1!}\cdot(-1)^{g-1}\Big[(-q)-2+(-q)^{-1}\Big]^{g-1}(1+q)^{\int_\beta c_1(X)}\bigg\}.
\nan 
Since by Lemma 3.1 in \cite{EQ}, the full primary Donaldson-Thomas theory is determined by those invariants with primary insertions (if any) of degree$>2$, it follows that the BPS state counts vanish if insertions of degree$<2$ appear, and they satisfy the divisor equation. 

Note that the Donaldson-Thomas theory counts curves only in a virtual sense. However, it is expected that BPS state counts are enumerative. More precisely, assume that $\gamma_1,\cdots,\gamma_m$ are integral, and let $X_i\subset X$ be a subvariety which is the Poinca\'e dual of $\gamma_i$ in general position. Then $n_{g,\beta}^X(\gamma_1,\cdots,\gamma_m)$ is expected to the number of irreducible embedded curves in $X$ of geometric genus $g$, with homology class $\beta$ and intersecting with all $X_i$'s.

To prove Corollary \ref{BPS}, we only need to consider insertions of degree$>2$. Without loss of generality, let $\{T_i\}_{0\leqslant i\leqslant L}$ be a basis of $H^{>2}(X,\bQ)$. Since $[C_1],\cdots,[C_l]$ generate an extremal face in $NE(X)$, it follows that
\ba
&&\suml_{\beta\in Cen(f)}v^\beta Z'_{DT}(X;q|)_\beta\nonumber\\
&=&\exp\bigg\{\suml_{\beta\in Cen(f)}v^\beta\suml_{g\in\bZ}\suml_{r\in div(\beta)}n_{g,\frac\beta r}^X\cdot\frac{(-1)^{g-1}}{r}\Big[(-q)^r-2+(-q)^{-r}\Big]^{g-1}\bigg\},\label{BPSXcenter}
\na
and then
\ba
&&\frac{\suml_{\substack{\beta\in H_2(X,\bZ)\\\int_\beta c_1(X)=0}}v^\beta Z'_{DT}(X;q|)_\beta+\suml_{\substack{\beta\in H_2(X,\bZ)\\\int_\beta c_1(X)>0}}v^\beta\suml_{e_0,\cdots,e_L\in\bZ_{\geqslant 0}}Z'_{DT}(X;q|\prodl_{i=0}^L\tilde\tau_0(T_i)^{e_i})_\beta\prodl_{i=0}^L\frac{t_i^{e_i}}{e_1!}}{\suml_{\beta\in Cen(f)}v^{\beta}Z'_{DT}(X;q|)_{\beta}}\nonumber\\
&=&\exp\bigg\{\suml_{\substack{\beta\in H_2(X,\bZ)\setminus Cen(f)\\\int_\beta c_1(X)=0}}v^\beta\suml_{g\in\bZ}\suml_{r\in div(\beta)}n_{g,\frac\beta r}^X\cdot\frac{(-1)^{g-1}}{r}\Big[(-q)^r-2+(-q)^{-r}\Big]^{g-1}\nonumber\\
&&\qquad+\suml_{\substack{\beta\in H_2(X,\bZ)\\\int_\beta c_1(X)>0}}v^\beta\suml_{g\in\bZ}\suml_{e_0,\cdots,e_L\in\bZ_{\geqslant 0}}n_{g,\beta}^X(\prodl_{i=0}^LT_i^{e_i})\prodl_{i=0}^L\frac{t_i^{e_i}}{e_1!}\nonumber\\
&&\qquad\qquad\cdot(-1)^{g-1}\Big[(-q)-2+(-q)^{-1}\Big]^{g-1}(1+q)^{\int_\beta c_1(X)}\bigg\}.\label{BPSX}
\na
Using the identification $\cF:H_2(X,\bZ)\xrightarrow{\cong}H_2(X',\bZ)$, we also have
\ba
&&\frac{\suml_{\substack{\beta\in H_2(X,\bZ)\\\int_{\cF\beta} c_1(X')=0}}v^\beta Z'_{DT}(X';q|)_{\cF\beta}+\suml_{\substack{\beta\in H_2(X,\bZ)\\\int_\beta c_1(X)>0}}v^\beta\suml_{e_0,\cdots,e_L\in\bZ_{\geqslant 0}}Z'_{DT}(X';q|\prodl_{i=0}^L\tilde\tau_0(\cF T_i)^{e_i})_{\cF\beta}\prodl_{i=0}^L\frac{t_i^{e_i}}{e_1!}}{\suml_{\beta\in Cen(f)}v^{\beta}Z'_{DT}(X';q|)_{\cF\beta}}\nonumber\\
&=&\exp\bigg\{\suml_{\substack{\beta\in H_2(X,\bZ)\setminus Cen(f)\\\int_\beta c_1(X)=0}}v^\beta\suml_{g\in\bZ}\suml_{r\in div(\beta)}n_{g,\frac{\cF\beta}{r}}^{X'}\cdot\frac{(-1)^{g-1}}{r}\Big[(-q)^r-2+(-q)^{-r}\Big]^{g-1}\nonumber\\
&&\qquad+\suml_{\substack{\beta\in H_2(X,\bZ)\\\int_{\cF\beta} c_1(X')>0}}v^\beta\suml_{g\in\bZ}\suml_{e_0,\cdots,e_L\in\bZ_{\geqslant 0}}n_{g,\cF\beta}^{X'}(\prodl_{i=0}^L(\cF T_i)^{e_i})\prodl_{i=0}^L\frac{t_i^{e_i}}{e_1!}\nonumber\\
&&\qquad\qquad\cdot(-1)^{g-1}\Big[(-q)-2+(-q)^{-1}\Big]^{g-1}(1+q)^{\int_{\cF\beta} c_1(X')}\bigg\}.\label{BPSX'}
\na
Note that 
\ban
\int_\beta c_1(X)=\int_{\phi_{w-1}^!\cdots\phi_0^!\beta}c_1(X_w)=\int_{(\phi'_{w-1})^!\cdots(\phi'_0)^!\cF\beta}c_1(X_w)=\int_{\cF\beta} c_1(X').
\nan
So \eqref{BPSflop} follows from \eqref{flop}, \eqref{BPSX} and \eqref{BPSX'}. 

For \eqref{BPScenter}, using the identification $-\cF:H_2(X,\bZ)\xrightarrow{\cong}H_2(X',\bZ)$, we get from \eqref{BPSXcenter}
\ba
&&\suml_{\beta\in Cen(f)}v^{-\beta} Z'_{DT}(X';q|)_{\cF\beta}\nonumber\\
&=&\exp\bigg\{\suml_{\beta\in Cen(f)}v^{-\beta}\suml_{g\in\bZ}\suml_{r\in div(\cF\beta)}n_{g,\frac{\cF\beta}{r}}^X\cdot\frac{(-1)^{g-1}}{r}\Big[(-q)^r-2+(-q)^{-r}\Big]^{g-1}\bigg\},\label{BPSX'center}
\na
So \eqref{BPScenter} follows from \eqref{center}, \eqref{BPSXcenter} and \eqref{BPSX'center}.

\section{GW/DT correspondence}

In this section, we give a proof of Corollary \ref{GWDT}. We first review basic materials in Gromov-Witten theory, and describe the change of Gromov-Witten theory under flops. Then we follow \cite{MNOP2} to recall the conjectural formulae for the GW/DT correspondence, and use these formulae to prove Corollary \ref{GWDT}. 

Let $\overline M_{g,m}(X,\beta)$ be the moduli space of $m$-pointed stable maps from connected, genus $g$ curves to $X$, representing the class $\beta\in H_2(X,\bZ)$. Let $ev_i:\overline M_{g,m}(X,\beta)\rightarrow X$ be the evaluation map at the $i$-th marked point, and set  
\ban
\psi_i:=c_1(L_i)\in H^2(\overline M_{g,m}(X,\beta),\bQ),
\nan
where $L_i$ is the cotangent line bundle associated to the $i$-th marked point. For $\gamma_1,\cdots,\gamma_m\in H^*(X,\bQ)$ and $d_1,\cdots,d_m\in\bZ_{\geqslant0}(m\geqslant0)$, define the (connected) correlator by
\ban
\<\prodl_{i=1}^m\tau_{d_i}(\gamma_i)\>_{g,\beta}^X:=\int_{[\overline M_{g,m}(X,\beta)]^{vir}}\prodl_{i=1}^m\psi_i^{d_i}ev_i^*(\gamma_i).
\nan

The conjectural GW/DT correspondence compares partition functions of disconnected Gromov-Witten invariants with reduced Donaldson-Thomas partition function. Let $\{T_i\}_{0\leqslant i\leqslant N}$ be a basis of $H^{*}(X,\bQ)$, and the disconnected partition functions in Gromov-Witten theory are given by the following identity:
\ban
&&1+\suml_{\beta\in H_2(X,\bZ)\setminus\{0\}}v^\beta\suml_{\substack{e_{d,i}\in\bZ_{\geqslant0}}}Z'_{GW}(X;u|\prodl_{\substack{d\geqslant0\\0\leqslant i\leqslant N}}\tau_d(T_i)^{e_{d,i}})_\beta\prodl_{\substack{d\geqslant0\\0\leqslant i\leqslant N}}\frac{t_{d,i}^{e_{d,i}}}{e_{d,i}!}\\
&=&\exp\bigg\{\suml_{\beta\in H_2(X,\bZ)\setminus\{0\}}v^\beta\suml_{g\in\bZ_{\geqslant 0}}u^{2g-2}\suml_{\substack{e_{d,i}\in\bZ_{\geqslant 0}}}\<\prodl_{\substack{d\geqslant0\\0\leqslant i\leqslant N}}\tau_d(T_i)^{e_{d,i}}\>_{g,\beta}^X\prodl_{\substack{d\geqslant0\\0\leqslant i\leqslant N}}\frac{t_{d,i}^{e_{d,i}}}{e_{d,i}!}\bigg\}.
\nan

For the change of Gromov-Witten theory under flops, we have the following theorem.
\begin{theorem}(Theorem A in \cite{LR})
Let $f$ be a general flop. Let $\gamma_1,\cdots,\gamma_m\in H^*(X,\bQ)$ and $d_1,\cdots,d_m\in\bZ_{\geqslant 0}(m\geqslant 0)$, such that $\gamma_i$ has support away from the center of $f$. Then 
\ba
&&\frac{1+\suml_{\beta\in H_2(X,\bZ)\setminus\{0\}}v^\beta Z'_{GW}(X;u|\prodl_{i=1}^m\tilde\tau_{d_i}(\gamma_i))_\beta}{1+\suml_{\beta\in Cen(f)\setminus\{0\}}v^{\beta} Z'_{GW}(X;u|)_{\beta}}\nonumber\\
&=&\frac{1+\suml_{\beta\in H_2(X,\bZ)\setminus\{0\}}v^\beta Z'_{GW}(X';u|\prodl_{i=1}^m\tilde\tau_{d_i}(\cF\gamma_i))_{\cF\beta}}{1+\suml_{\beta\in Cen(f)\setminus\{0\}}v^{\beta} Z'_{GW}(X';u|)_{\cF\beta}},\label{GWflop}\\
Z'_{GW}(X;u|)_\beta&=&Z'_{GW}(X';u|)_{-\cF\beta},\qquad\forall\beta\in Cen(f)\setminus\{0\}.\label{GWcenter}
\na
\end{theorem}
\begin{remark}
Theorem A in \cite{LR} only deals with the case $d_1=\cdots=d_m=0$, but the generalization is straightforward.
\end{remark}

Now we give precise formulae for the conjectural GW/DT correspondence. For primary insertions, we have the following conjecture.
\begin{conjecture}(Conjecture $2$ in \cite{MNOP2})
Suppose that $\gamma_1,\cdots,\gamma_m\in H^*(X,\bQ)(m\geqslant0)$. Then after the change of variables $q=-e^{\sqrt{-1}u}$, we have
\ban
(-\sqrt{-1}u)^{\int_\beta c_1(X)}Z'_{GW}(X;u|\prodl_{i=1}^m\tau_0(\gamma_i))_\beta&=&(-q)^{-\frac12\int_\beta c_1(X)}Z'_{DT}(X;q|\prodl_{i=1}^m\tilde\tau_0(\gamma_i))_\beta.
\nan
\end{conjecture}

The authors of \cite{MNOP2} conjectured that the descendent Gromov-Witten theory of $X$ is equivalent to the descendent Donaldson-Thomas theory of $X$ in a subtle way. In the general case, they did not find a complete formula for the conjectural correspondence. However, we have the following precise conjecture for the descendants of the point class. 

\begin{conjecture}(Conjecture $4'$ in \cite{MNOP2})
Let $P$ be the class of a point in $X$. Suppose that $\gamma_1,\cdots,\gamma_m\in H^{>0}(X,\bQ)(m\geqslant0)$, and $d_1,\cdots,d_n\in\bZ_{\geqslant 0}(n\geqslant0)$. Then after the change of variables $q=-e^{\sqrt{-1}u}$, we have
\ban
&&(-\sqrt{-1}u)^{\int_\beta c_1(X)-\suml_{i=1}^nd_i}Z'_{GW}(X;u|\prodl_{i=1}^m\tau_0(\gamma_i)\prodl_{i=1}^n\tau_{d_i}(P))_\beta\\
&=&(-q)^{-\frac12\int_\beta c_1(X)}Z'_{DT}(X;q|\prodl_{i=1}^m\tilde\tau_0(\gamma_i)\prodl_{i=1}^n\tilde\tau_{d_i}(P))_\beta.
\nan
\end{conjecture}

To prove Corollary \ref{GWDT}, note that by Lemma 3.1 in \cite{EQ}, we only need to consider insertions whose pullback classes have degree$>2$. Without loss of generality, let $\{T_i\}_{0\leqslant i\leqslant L}$ be a basis of $H^{*}(X,\bQ)$, where $T_0$ is the class of a point. Then the assumption in Corollary \ref{GWDT} gives
\ban
1+\suml_{\beta\in H_2(X,\bZ)\setminus\{0\}}v^\beta(-\sqrt{-1}u)^{\int_\beta c_1(X)}\suml_{\substack{e_{0,1},\cdots,e_{0,L}\in\bZ_{\geqslant 0}\\e_{d,0}\in\bZ_{\geqslant0}}}Z'_{GW}(X;u|\prodl_{i=1}^L\tau_0(T_i)^{e_{0,i}}\prodl_{d=0}^\infty\tau_d(T_0)^{e_{d,0}})_\beta\prodl_{i=1}^L\frac{t_{0,i}^{e_{0,i}}}{e_{0,i}!}\prodl_{d=0}^\infty\frac{((-\sqrt{-1}u)^{-1}t_{d,0})^{e_{d,0}}}{e_{d,0}!}\nonumber\\
=\suml_{\beta\in H_2(X,\bZ)}v^\beta(-q)^{-\frac12\int_\beta c_1(X)}\suml_{\substack{e_{0,1},\cdots,e_{0,L}\in\bZ_{\geqslant 0}\\e_{d,0}\in\bZ_{\geqslant 0}}}Z'_{DT}(X;q|\prodl_{i=1}^L\tilde\tau_0(T_i)^{e_{0,i}}\prodl_{d=0}^\infty\tilde\tau_d(T_0)^{e_{d,0}})_\beta\prodl_{i=1}^L\frac{t_{0,i}^{e_{0,i}}}{e_{0,i}!}\prodl_{d=0}^\infty\frac{t_{d,0}^{e_{d,0}}}{e_{d,0}!}.\label{GWDTX}
\nan
Note that the map $v^\beta\mapsto v^\beta(-\sqrt{-1}u)^{\int_\beta c_1(X)}$ gives an isomorphism in the Novikov ring of $X$, and then \eqref{GWflop} implies that
\ban
\frac{1+\suml_{\beta\in H_2(X,\bZ)\setminus\{0\}}v^\beta(-\sqrt{-1}u)^{\int_\beta c_1(X)}\suml_{\substack{e_{0,1},\cdots,e_{0,L}\in\bZ_{\geqslant 0}\\e_{d,0}\in\bZ_{\geqslant0}}}Z'_{GW}(X;u|\prodl_{i=1}^L\tau_0(T_i)^{e_{0,i}}\prodl_{d=0}^\infty\tau_d(T_0)^{e_{d,0}})_\beta\prodl_{i=1}^L\frac{t_{0,i}^{e_{0,i}}}{e_{0,i}!}\prodl_{d=0}^\infty\frac{((-\sqrt{-1}u)^{-1}t_{d,0})^{e_{d,0}}}{e_{d,0}!}}{1+\suml_{\beta\in Cen(f)\setminus\{0\}}v^{\beta}Z'_{GW}(X;u|)_{\beta}}\\
=\frac{1+\suml_{\beta\in H_2(X,\bZ)\setminus\{0\}}v^\beta(-\sqrt{-1}u)^{\int_\beta c_1(X)}\suml_{\substack{e_{0,1},\cdots,e_{0,L}\in\bZ_{\geqslant 0}\\e_{d,0}\in\bZ_{\geqslant0}}}Z'_{GW}(X';u|\prodl_{i=1}^L\tau_0(\cF T_i)^{e_{0,i}}\prodl_{d=0}^\infty\tau_d(\cF T_0)^{e_{d,0}})_{\cF\beta}\prodl_{i=1}^L\frac{t_{0,i}^{e_{0,i}}}{e_{0,i}!}\prodl_{d=0}^\infty\frac{((-\sqrt{-1}u)^{-1}t_{d,0})^{e_{d,0}}}{e_{d,0}!}}{1+\suml_{\beta\in Cen(f)\setminus\{0\}}v^{\beta}Z'_{GW}(X';u|)_{\cF\beta}}.
\nan
Here we have used the change of variables $t_{d,0}\mapsto(-\sqrt{-1}u)^{-1}t_{d,0}$. Similarly, note that the change of variables $v^\beta\mapsto v^\beta(-q)^{-\frac12\int_\beta c_1(X)}$ gives an isomorphism of the Novikov ring of $X$, and then \eqref{flop} gives
\ban
\frac{\suml_{\beta\in H_2(X,\bZ)}v^\beta(-q)^{-\frac12\int_\beta c_1(X)}\suml_{\substack{e_{0,1},\cdots,e_{0,L}\in\bZ_{\geqslant 0}\\e_{d,0}\in\bZ_{\geqslant 0}}}Z'_{DT}(X;q|\prodl_{i=1}^L\tilde\tau_0(T_i)^{e_{0,i}}\prodl_{d=0}^\infty\tilde\tau_d(T_0)^{e_{d,0}})_\beta\prodl_{i=1}^L\frac{t_{0,i}^{e_{0,i}}}{e_{0,i}!}\prodl_{d=0}^\infty\frac{t_{d,0}^{e_{d,0}}}{e_{d,0}!}}{\suml_{\beta\in Cen(f)}v^{\beta}Z'_{DT}(X;q|)_{\beta}}\\
=\frac{\suml_{\beta\in H_2(X,\bZ)}v^\beta(-q)^{-\frac12\int_\beta c_1(X)}\suml_{\substack{e_{0,1},\cdots,e_{0,L}\in\bZ_{\geqslant 0}\\e_{d,0}\in\bZ_{\geqslant 0}}}Z'_{DT}(X';q|\prodl_{i=1}^L\tilde\tau_0(\cF T_i)^{e_{0,i}}\prodl_{d=0}^\infty\tilde\tau_d(\cF T_0)^{e_{d,0}})_{\cF\beta}\prodl_{i=1}^L\frac{t_{0,i}^{e_{0,i}}}{e_{0,i}!}\prodl_{d=0}^\infty\frac{t_{d,0}^{e_{d,0}}}{e_{d,0}!}}{\suml_{\beta\in Cen(f)}v^{\beta}Z'_{DT}(X';q|)_{\cF\beta}}.
\nan
Now from \eqref{center}, \eqref{GWcenter} and the assumption in Corollary \ref{GWDT}, we obtain
\ban
Z'_{GW}(X';u|)_{\cF\beta}=Z'_{GW}(X;u|)_{-\beta}=Z'_{DT}(X;q|)_{-\beta}=Z'_{DT}(X';q|)_{\cF\beta},\quad\forall\beta\in Cen(f)\setminus\{0\},
\nan
and then the desired result follows from the above three long equalities.

{\bf Acknowledgements.} 

The author would like to thank Yongbin Ruan, Jianxun Hu, Wei-Ping Li and Zhenbo Qin for many constructive discussions, and Jian Zhou for warm encouragement. The author would also like to thank Xiaowen Hu for providing him the referernce \cite{Pi}.


\begin{thebibliography}{999}

\bibitem[Ca]{Ca}Calabrese, J., Donaldson-Thomas invariants and flops, arxiv: 1111.1670.

\bibitem[Cl]{Cl}Clemens, C. H., Degeneration of K\"ahler manifolds, Duke Math. J., 44 (2)(1977), 215-290.


\bibitem[DT]{DT}Donaldson, S., Thomas, R., Gauge theory in higher dimensions, in The Geometric Universe: Science, Geometry, and the Work of Roger Penrose, S. Huggett et. al eds., Oxford Univ. Press, (1998).

\bibitem[EQ]{EQ}Edidin, D., Qin, Z., The Gromov-Witten and Donaldson-Thomas correspondence for trivial elliptic fibrations, Internat. J. Math. 18(7) (2007), 821-838. 

\bibitem[GV1]{GV1}Gopakumar, R., Vafa, C., M-theory and topological strings I, hep-th/9809187.

\bibitem[GV2]{GV2}Gopakumar, R., Vafa, C., M-theory and topological strings II, hep-th/9812127.

\bibitem[HHKQ]{HHKQ}He, W., Hu, J., Ke, H., Qi, X., Blow-up formulae of high genus Gromov-Witten invariants in dimension six. arXiv:1402.4221

\bibitem[HL]{HL}Hu, J., Li, W.-P., The Donaldson-Thomas invariants under blowups and flops, J. Differential Goem. 90 (3)(2012), 391-411.

\bibitem[HLR]{HLR}Hu, J., Li, T.-J., Ruan, Y., Birational cobordism invariance of uniruled symplectic manifolds, Invent. Math. 172(2008), 231-275.

\bibitem[Ke]{Ke}Ke, H.-Z., Stable pair invariants under blow-ups, Math. Z. DOI 10.1007/s00209-015-1568-7.

\bibitem[Ko]{Ko}Koll\'ar, J., Flips, flops, minimal models, etc., Surveys in differential geometry (Cambridge, MA, 1990), Lehigh Univ., Bethlehem, PA, (1991), 113-199.

\bibitem[KM]{KM}Koll\'ar, J., Mori, S., Birational geometry of algebraic varieties (with the collaboration of C. H. Clemens and A. Corti). Translated from the 1998 Japanese original. Cambridge Tracts in Mathematics, 134. Cambridge University Press, Cambridge, 1998.

\bibitem[La]{La}Laufer, H., On $\bC\bP^1$ as an exceptional set, Recent developments in several complex variables,
Ann. of Math. Stud. Princeton University Press, 100 (1981), 261-275.

\bibitem[LR]{LR}Li, A.-M., Ruan, Y., Symplectic surgery and Gromov-Witten invariants of Calabi-Yau 3-folds, Invent. Math. 145 (1)(2001), 151-218.

\bibitem[LW]{LW}Li, J., Wu, B., Good degeneration of quot-schemes and coherent systems, Comm. Anal. Geom. 23(4) (2015), 841-921. 

\bibitem[MNOP1]{MNOP1}Maulik, D., Nekrasov, N., Okounkov, A., Pandharipande, R., Gromov-Witten theory and Donaldson-Thomas theory I, Compositio Math. 142 (2006), 1263-1285.

\bibitem[MNOP2]{MNOP2}Maulik, D., Nekrasov, N., Okounkov, A., Pandharipande, R., Gromov-Witten theory and Donaldson-Thomas theory I, Compositio Math. 142 (2006), 1286-1304.

\bibitem[MOOP]{MOOP}Maulik, D., Oblomkov, A., Okounkov, A., Pandharipande, R., Gromov-Witten/Donaldson-Thomas correspondence for toric 3-folds, Invent. Math. 186(2) (2011), 435-479.

\bibitem[MP]{MP}Maulik, D., Pandharipande, R., A topological view of Gromov-Witten theory, Topology 45(2006), 887-918.


\bibitem[Na]{Na}Nakajima, H., Lectures on Hilbert schemes of points on surfaces,  University Lecture Series, 18. American Mathematical Society, Providence, RI, 1999.

\bibitem[P1]{P1}Pandharipande, R., Hodge integrals and degenerate contributions, Commun. Math. Phys. 208(1999), 489-506.

\bibitem[P2]{P2}Pandharipande, R., Three questions in Gromov-Witten theory. Proceedings of the International Congress of Mathematicians, Vol. II (Beijing, 2002), 503-512, Higher Ed. Press, Beijing, 2002. 

\bibitem[Pi]{Pi}Pinkham, H., Factorization of birational maps in dimension 3, Singularities (P. Orlik, ed.), Proc. Symp. Pure Math., vol. 40, Part 2, American Mathematical Society, Providence, 1983, pp. 343-371.

\bibitem[PP]{PP}Pandharipande, R., Pixton, A., Gromov-Witten/Pairs correspondence for the quintic $3$-fold, arXiv:1206.5490v1.

\bibitem[PT]{PT}Pandharipande, R., Thomas, R. P., Curve counting via stable pairs in the derived category, Invent. Math. 178(2) (2009), 407-447.

\bibitem[Re]{Re}Reid, M., Minimal models of canonical $3$-folds, Algebraic varieties and analytic varieties (Tokyo, 1981), Adv. Stud. Pure Math., 1, North-Holland, Amsterdam, (1983), 131-180.

\bibitem[Ru]{Ru}Ruan, Y., Surgery, quantum cohomology and birational geometry, Northern California Symplectic Geometry Seminar AMS Translations, Series 2, 196(1999) , 183-198.

\bibitem[Th]{Th}Thomas, R., A holomorphic Casson invariant for Calabi-Yau 3-folds and bundles on K3 fibrations, JDG 53 (1999), 367-438.

\bibitem[T1]{T1}Toda, Y., Curve counting theories via stable objects I. DT/PT correspondence, J. Amer. Math. Soc. 23 (4)(2010), 1119-1157.

\bibitem[T2]{T2}Toda, Y., Curve counting theories via stable objects II: DT/ncDT flop formula, J. Reine Angew. Math. 675 (2013), 1-51. 



\end{thebibliography}
\end{document}